\documentclass{amsart}

\usepackage{appendix}
\usepackage{amsmath}
\usepackage{amsfonts}
\usepackage{amsthm}
\usepackage{tikz}
\usepackage[inline]{asymptote}
\usepackage{asymptote}
\usepackage{graphicx}
\usepackage{listings}
\usepackage{amssymb}
\usepackage[T1]{fontenc}
\usepackage[normalem]{ulem}
\usepackage{hyperref}
\usepackage{caption}
\usepackage[section]{algorithm}
\usepackage{algpseudocode}
\usepackage[tableposition=top]{caption}
\usepackage{multirow}
\usepackage{verbatim}
\allowdisplaybreaks
\newcounter{nalg}[section] 
\renewcommand{\thenalg}{\thesection .\arabic{nalg}} 
\DeclareCaptionLabelFormat{algocaption}{\textbf{Algorithm \thenalg}} 

\numberwithin{equation}{section}
\usepackage{footmisc}
\hypersetup{colorlinks = true}
\theoremstyle{definition}

\usepackage{caption}
\usepackage[section]{algorithm}
\usepackage{algpseudocode}

\usepackage{mathtools}

\renewcommand{\thenalg}{\thesection .\arabic{nalg}} 
\DeclareCaptionLabelFormat{algocaption}{\textbf{Algorithm \thenalg}} 

\usepackage{ltablex}

\definecolor{mygreen}{rgb}{0,0.6,0}
\definecolor{mygray}{rgb}{0.5,0.5,0.5}
\definecolor{mymauve}{rgb}{0.58,0,0.82}

\numberwithin{equation}{section}

\newtheorem{lem}[equation]{Lemma}

\newtheorem{theorem}[equation]{Theorem}
\newtheorem{definition}[equation]{Definition}

\newtheorem{prop}[equation]{Proposition}
\newtheorem{corollary}[equation]{Corollary}
\newtheorem{question}[equation]{Question}
\newtheorem{conjecture}[equation]{Conjecture}

\newcommand{\BQ}{\mathbf{Q}}
\newcommand{\BZ}{\mathbf{Z}}
\newcommand{\BC}{\mathbf{C}}
\newcommand{\BF}{\mathbf{F}}
\newcommand{\BN}{\mathbf{N}}

\newcommand{\Aut}{\mathrm{Aut}}

\usepackage{verbatim}
\usepackage[margin=1in]{geometry}

\title{Ramification of Wild Automorphisms of Laurent Series Fields}
\newtheorem*{remark}{Remark}

\author{Kenz Kallal}
\address{Department of Mathematics, Harvard University, Cambridge, MA 02138}
\email{kenzkallal@college.harvard.edu}

\author{Hudson Kirkpatrick}
\address{Department of Mathematics, Princeton University, Princeton, NJ 08544-1000}
\email{Hbk@princeton.edu}

\date{March 22, 2019}

\begin{document}
	
	\maketitle
	\begin{abstract}
		Let $K$ be a complete discrete valuation field with residue class field $k$, where both are of positive characteristic $p$. Then the group of wild automorphisms of $K$ can be identified with the group under composition of formal power series over $k$ with no constant term and $X$-coefficient $1$. Under the hypothesis that $p > b^2$, we compute the first nontrivial coefficient of the $p$th iterate of a power series over $k$ of the form $f = X + \sum_{i \geq 1} a_iX^{b+i}$. As a result, we obtain a necessary and sufficient condition for an automorphism to be ``$b$-ramified,'' having lower ramification numbers of the form $i_n(f) = b(1 + \cdots + p^n)$. This is a vast generalization of Nordqvist's 2017 theorem on $2$-ramified power series, as well as the analogous result for minimally ramified power series which proved to be useful for arithmetic dynamics in a 2013 paper of Lindahl on linearization discs in $\BC_p$ and a 2015 result of Lindahl--Rivera-Letelier on optimal cycles over nonarchimedean fields of positive residue characteristic. The success of our computation is also promising progress towards a generalization of Lindahl--Nordqvist's 2018 theorem bounding the norm of periodic points of $2$-ramified power series.
	\end{abstract}
	
	\section{Introduction}
	Let $k$ be a field of characteristic $p > 0$, and consider the complete discrete valuation field $K = k((X))$ under the $X$-adic valuation. Any $k$-linear automorphism of $K$ is continuous, and therefore determined by a choice of the image of $X$ in the set of non-units of $k[[X]]$. The group of \emph{wild} automorphisms, the subgroup of $\Aut(K/k)$ consisting of all $\sigma$ such that $\nu_X(\sigma(X) - X) > 1$, can therefore be identified with the following group of formal power series over $k$. 
	\begin{definition}[$\mathcal{N}(k)$]
		We write $\mathcal{N}(k)$ to denote the group under composition of formal power series of the form \[f(X) = X + \sum_{i=2}^\infty a_iX^i\] with coefficients $a_i \in k$.
	\end{definition}
	\begin{remark}
		In the case $k = \BF_p$, the group $\mathcal{N}(k)$ is called the \emph{Nottingham group}. It is an example of a pro-$p$ group, and is thus of general interest in both group theory and number theory (see \cite[Ch. 1]{serre} and \cite[Ch. 10]{book}). It was first introduced by Jennings \cite{inspiration}, Johnson  \cite{johnson} and York \cite{york2, york1}. Of particular importance to group theory are some results of Leedham-Green--Weiss and Camina \cite[Theorems 3 and 5]{subgroups}; they state that every finite $p$-group can be embedded in $\mathcal{N}(\BF_p)$, and every countably-based pro-$p$ group can be embedded as a closed subgroup of $\mathcal{N}(\BF_p)$. Other important group-theoretic considerations about the Nottingham group are those about torsion: Klopsch \cite{klopsch} classified all elements of $\mathcal{N}(\BF_p)$ of order $p$ up to conjugacy, and his result was generalized by Jean \cite{finite1} and Lubin \cite{finite2} to order $p^n$ (see also \cite{finite3}).         
	\end{remark}
	In this paper we always use the notation $f^n$ to mean the $n$-th iterate of $f \in \mathcal{N}(k)$. By taking compositional inverses, this notation is well-defined for any $n \in \BZ$. If $G \subseteq \Aut(K/k)$ is the group of wild $k$-linear automorphisms, then the identification $\mathcal{N}(k) \xrightarrow{\sim} G$ is given by mapping $f(X)$ to the automorphism $g \mapsto g \circ f^{-1}$. In this way, composition and inverses of automorphisms correspond to the same operations for power series (where the inverses are compositional inverses). This gives a convenient explicit way to describe the higher ramification groups of $K/k$. For a given power series $f \in \mathcal{N}(k)$, we are interested in the ramification numbers with respect to the filtration of ramification groups of $\Aut(K/k)$ of the $\sigma \in \Aut(K/k)$ corresponding to $f$ and its iterates.
	From the identification between $\mathcal{N}(k)$ and the group of wild automorphisms of $K$, the ramification number of $f$ should be a constant factor from $\nu_X(f^{-1}(X) - X)$. Since the ramification groups are groups, the ramification number of $f$ is the same as that of $f^{-1}$. Consistent with this, we can define the ramification numbers explicitly in terms of power series without worrying about taking inverses:
	\begin{definition}[Ramification]\label{ramification}
		The \emph{ramification number} of $f \in \mathcal{N}(k)$ is \[i(f) := \nu_X(f(X) - X) - 1.\]
	\end{definition}
	Note that this normalizes the ramification number so that the least ramified $f \in \mathcal{N}(k)$ have $i(f) = 1$ instead of $2$.

	We will frequently mention the ``$n^\text{th}$ ramification number'' of $f$, that is, $i_n(f) := i\left(f^{p^n}\right)$. In fact, we are interested in power series with certain types of sequences of ramification numbers.
	\begin{definition}
		Let $f \in \mathcal{N}(k)$ and $b \geq 1$. We write that $f$ is $b$-\emph{ramified} if it has ramification numbers of the form
		\[i_n(f) = b(1 + p + \cdots + p^n)\]
		for all $n \geq 0$. 
	\end{definition}
	We will justify the choice of ramification of this form with Theorem~\ref{hammer}, a powerful result of Laubie and Sa\"ine which restricts it to that form under certain conditions. We will also see that $b$-ramification is the minimal sequence of ramification numbers given $i(f) = b$. The main question that this paper is concerned with is the following:
	\begin{question}\label{classify}
		Let $b \in \BN$ with $p \nmid b$. Given a sequence $\{a_i\}_{i \geq 1}$ of coefficients in $k$, is the formal power series \[f(X) = X + \sum_{i=1}^\infty a_iX^{i+b} \in \mathcal{N}(k)\] $b$-ramified?
	\end{question}
	We will call $a_1, \dots, a_n$ as defined in Question~\ref{classify} the ``first $n$ nontrivial coefficients of $f$.''  Nordqvist \cite{swedish} solved Question~\ref{classify} for $b = 2$ by giving a polynomial condition on finitely many of the $a_i$'s. This itself is a generalization of the characterization of minimally ramified power series (see \cite[Exemple 3.19]{letelier} or somewhat more generally \cite[Theorem E]{lindahl_l}). The computations that lead to these results have been relevant to many important theorems in nonarchimedean dynamical systems \cite{dynamic1, dynamic2, lindahl, lindahl_l, letelier, generic}.
	
	The main result of this paper generalizes the polynomial condition given by Nordvist to the following under certain hypotheses:
	\begin{definition}[$P_b$]\label{P}
		For $b \in \BN$, define the polynomial in $\BZ[x_1, \ldots, x_{b+1}]$
		\[P_b(x_1, \ldots, x_{b+1}) := (b+1)x_1^{b+1} + 2\hspace{-1em}\sum_{(e_1, \ldots, e_{b+1})}(-1)^{e_2 + \cdots + e_{b+1}}\frac{(e_2 + \cdots + e_{b+1})!}{e_2! \cdots e_{b+1}!}x_1^{e_1}\cdots x_{b+1}^{e_{b+1}}\]
		where the sum is taken over all tuples $(e_1, \ldots, e_{b+1}) \in \BZ^{b+1}_{\geq 0}$ such that $\sum_{i=1}^{b+1}e_i = b$ and $\sum_{i=1}^{b+1}ie_i = 2b$.
	\end{definition}
	This polynomial is what results from a computation of the first nontrivial coefficient of $f^p$ given those of $f$. 
	
	The main result of this paper is the following theorem\footnote{Since we posted on arXiv.org a draft of this paper containing this result, Jonas Nordqvist and Juan Rivera-Letelier informed us in a private communication that they have independently proved a version of it using the language of iterative residues.}, answering Question~\ref{classify} in the case where $b^2 < p$.
	\begin{theorem}\label{mainresult}
		Fix an odd prime $p$ and $b \in \BN$ such that $b^2 < p$, and let $k$ be an arbitrary field of characteristic $p$. Let $f(X) = X + \sum_{i=1}^\infty a_iX^{i+b} \in \mathcal{N}(k)$. Then $f$ is $b$-ramified if and only if $a_1 \neq 0$ and $P_b(a_1, \ldots, a_{b+1}) \neq 0$.
	\end{theorem}
	Note that the resulting
	\begin{align*}
	P_1 &= 2x_1^2 - 2x_2 \\
	P_2 &= 3x_1^3 + 2x_2^2 - 2x_1x_3
	\end{align*}
	agree with the previous results of Rivera-Letelier~\cite[Exemple 3.19]{letelier} and Nordqvist~\cite[Theorem 1]{swedish} answering Question~\ref{classify} in the cases $b = 1, 2$.

	This paper is organized as follows. In Section~\ref{fransson}, we describe the historical results on ramification numbers of power series which are necessary for our work. In Section~\ref{computation}, we give an explicit computation of the first nontrivial coefficient of the $p$-th iterate of a power series $f \in \mathcal{N}(k)$ with $i(f) = b$, which yields the desired necessary and sufficient condition for $f$ being $b$-ramified. In Section~\ref{dynamics}, we describe the implications this computation has for a generalization of Lindahl--Nordqvist's work in \cite{periodicpts} on the locations of periodic points in the open unit disc of a nonarchimedean field of characteristic $p$ under iteration of power series. 
	
	\section{Ramification of Power Series}\label{fransson}
	In this section and the next, $k$ is taken, as usual, to be an arbitrary field of characteristic $p > 0$. It is useful to have strong restrictions on the possible sequences of ramification numbers of elements of $\mathcal{N}(k)$. The most important historical result with regard to this is due to Sen \cite[Theorem~1]{sen}:
	\begin{theorem}[Sen, 1969]\label{senthm}
		Let $f \in \mathcal{N}(k)$, where $k$ is a field of characteristic $p > 0$. Then for all $n \in \BN$, \[i_n(f) \equiv i_{n-1}(f) \pmod{p^n}.\]
	\end{theorem}
	Sen originally stated Theorem~\ref{senthm} under the hypothesis that $k$ is also perfect, but it holds in the more general setting by an argument due to Lubin \cite[pg. 64]{lubin_sen}.
	The following corollary immediately follows from power series computations and Sen's Theorem.
	\begin{corollary}\label{ineq}
		Let $f \in \mathcal{N}(k)$. Then for all $n \geq 0$, \begin{equation}\label{sen_ineq} i_n(f) \geq 1 +  p + \dots + p^n.\end{equation}
	\end{corollary}
	
	Power series $f \in \mathcal{N}(k)$ for which equality holds in (\ref{sen_ineq}) are called \emph{minimally ramified}. Another useful result is a theorem of Keating \cite[Theorem~7]{keating}, which gives a general form for $i_n(f)$ given the first two ramification numbers under certain conditions. 
	\begin{theorem}[Keating, 1992]\label{mah_boi}
		Let $f \in \mathcal{N}(k)$. If $i_0(f) = 1$ and $i_1(f) = 1 + bp$ with $1\leq b\leq p -2$, then $i_n(f) = 1 + bp + \cdots + bp^n$.
	\end{theorem}
	We will use a powerful generalization of Theorem~\ref{mah_boi} from \cite[Theorem~2]{laubie_saine}.
	\begin{theorem}[Laubie and Sa\"{i}ne, 1997]\label{hammer}
		Let $f \in \mathcal{N}(k)$. Then the following hold:
		\begin{enumerate}
			\item If $p | i(f)$, then $i_n(f) = p^ni(f)$ for all $n \geq 0$.
			\item Otherwise, if $i_1(f) < (p^2 - p + 1)i(f)$, then \[i_n(f) = i(f) + (1 + \dots + p^{n-1})(i_1(f) - i(f))\]
			for all $n \geq 1$. 
		\end{enumerate}
	\end{theorem}
	It follows from Theorem~\ref{hammer} that if $f \in \mathcal{N}(k)$ with $i(f) = b$, $i_1(f) = b + bp$ and $p \nmid b$, then $f$ is $b$-ramified. Since $b$-ramified power series must have $i(f) = b$, to answer Question~\ref{classify} it suffices to determine the power series $f = X + \sum_{i =1}^{\infty}a_iX^{i+b}$ with $a_1 \neq 0$ with the property that $i(f^p) = b + bp$.
	
	The special case $b=2$ was solved using exactly this technique in \cite[Theorem~1]{swedish}:
	\begin{theorem}[Nordqvist, 2017] \sloppy
		A power series given by $f(X) = X + \allowbreak \sum_{i = 1}^{\infty} a_iX^{i+2} \in \mathcal{N}(k)$ is $2$-ramified if and only if $3a_3^3 + 2a_4^2 - 2a_2a_5 \neq 0$ and $a_1 \neq 0$.
	\end{theorem}
	We will prove a similar result for the $b$-ramified power series in $\mathcal{N}(k)$ when $p > b^2$. 
	
	\section{Classification of \texorpdfstring{$b$}{\emph{b}}-Ramified Power Series}\label{computation}
	Let $b \in \BN$, and suppose $p > b^2$. 
	Let 
	\[f(X) = X + a_1X^{b+1} + a_2X^{b+2} + \dots \in \mathcal{N}(k),\] where $a_1 \neq 0$ so that $i_0(f) = b$. By Theorem~\ref{hammer}, $f$ is $b$-ramified if and only if $i_1(f) = b + bp$. Therefore, it suffices to compute the coefficients on the terms of degree at most $b + bp + 1$ in $f^{p}$ in terms of the $a_i$'s. To do this, we rely on a technique established in \cite[Lemma 3.6]{lindahl_l}. In particular, define the power series $\{\Delta_m\}_{m \geq 1}$ recursively by $\Delta_1 = f(X) - X$ and 
	\[\Delta_{m+1} = \Delta_m \circ f - \Delta_m.\]
	It is straightforward to check the following lemma by induction:
	\begin{lem}Suppose $m \in \BN$. Then, 
		\[\Delta_m(X) = \sum_{i = 0}^m (-1)^i\binom{m}{i}f^{m-i}(X).\]
	\end{lem}
	The following corollary is immediate by elementary properties of binomial coefficients modulo $p$:
	\begin{corollary}\label{d_p}
		$\Delta_p(X) = f^{p}(X) - X$.
	\end{corollary}
	Therefore, the desired quantity is
	\[i(f^p) = \nu_X(\Delta_p) - 1,\]
	so it suffices to compute the coefficients of $\Delta_p$. In particular, we wish to compute $\Delta_p$ modulo $X^{b + bp + 2}$, since this will provide us with all the coefficients of $f$ on the terms of degree at most $b + bp + 1$, as desired. 
	
	To do this, we prove a simple lemma that allows us to set up the necessary computational infrastructure.
	\begin{lem}\label{deltaval}
		$\nu_X(\Delta_m) \geq bm + 1$ for all $m \in \mathbf{N}$.
	\end{lem}
	\begin{proof}
		We proceed by induction. For the base case, we have already that 
		\[\Delta_1 = f(X) - X = a_1X^{b+1} + \cdots,\]
		so $\nu_X(\Delta_1) = b+1$ as desired. Now, assume the inductive hypothesis that $\nu_X(\Delta_m) \geq bm + 1$ for some $m \in \BN$. This means that we can write
		\[\Delta_m = A_1(m)X^{bm+1} + A_2(m)X^{bm+2} + \cdots\]
		for some $A_1(m), \ldots \in \BF_p$. From here it is clear (from the definition of $f$) that the term of least degree in \[\Delta_{m+1}(X) = (A_1(m)f(X)^{bm+1} + A_2(m)f(X)^{bm+2} + \cdots) - (A_1(m)X^{bm+1} + A_2(m)X^{bm+2} + \cdots)\]
		has degree at least $b(m+1) + 1$ [in particular the term of least degree that is not guaranteed to be cancelled by subtracting $\Delta_m$ is produced by the term of $A_1(m)f(X)^{bm+1}$ produced by choosing $bm$ copies of $X$ and one copy of $X^{b+1}$]. The desired result follows by induction. 
	\end{proof}
	\begin{corollary}
		$i_1(f) \geq bp$.
	\end{corollary}
	\begin{corollary}\label{sencor}
		$i_1(f) \geq bp + b$.
	\end{corollary}
	\begin{proof}
		By Sen's Theorem, $i_1(f) \equiv i_0(f) = b \pmod{p}$, which means that $i_1(f) = b + np$ for some integer $n$. If $n < b$, then by the assumption that $p > b$ (recall we actually assumed the stronger bound $p > b^2$), we have \[i_1(f) = b + np \leq b + (b-1)p = bp + (b-p) < bp,\] which contradicts the previous corollary.
	\end{proof}
	Corollary~\ref{sencor} means that under the hypothesis that $p > b$, $b$-ramification is the minimal possible ramification type subject to the condition $i_0(f) = b$.

	By Lemma~\ref{deltaval}, there are no terms of degree at most $bp$ in $\Delta_p$. Thus, we only need to consider the terms of degree $d$ for $bp + 1 \leq d \leq b + bp + 1$. In the language of the proof of the lemma, we need only to compute the coefficients $A_1(p), \ldots, A_{b+1}(p)$ in terms of the $a_i$'s. But Corollary~\ref{sencor} guarantees that $A_1(p), \ldots, A_b(p)$ are identically zero. It remains to compute $A_{b+1}(p)$ in terms of the $a_i$'s, which will in turn give us a criterion for $b$-ramification: by Theorem~\ref{hammer}, $f$ is $b$-ramified if and only if $A_{b+1}(p) \neq 0$. To do this, we compute using the recursive definition of the $A_i(m)$'s. In particular, we can compute modulo $X^{b + b(m+1) + 2}$ that since  $f(X)^{n}$ for $n > b + bm + 1$ can contribute no term that isn't cancelled by subtracting $\Delta_m(X)$,
	\begin{align*}\Delta_{m+1}(X) &\equiv A_1(m)f(X)^{bm+1} + A_2(m)f(X)^{bm+2} + \cdots + A_{b+1}f(X)^{b + bm + 1} - \Delta_m(X)\\
	&\equiv \bigg[A_1(m)a_1(bm+1))X^{b(m+1) + 1} + (A_1(m)a_2(bm+1) + A_2(m)a_1(bm+2)\bigg]X^{b(m+1) + 2}\\
	&+ \cdots + \bigg[A_1(m)a_b(bm+1) + \cdots + A_b(m)a_1(bm + b)\bigg]X^{b(m+1) + b}\\
	&+ \hspace{-0.15em}\left[\hspace{-0.25em}A_1(m)a_{1}^2\binom{bm+1}{2} \hspace{-0.25em} + \hspace{-0.25em} A_1(m)a_{b+1}(bm+1) + \cdots + A_{b+1}(m)a_1(bm + b + 1)\right]X^{b(m+1) + b + 1}.
	\end{align*}
	Recalling from the definition of $\Delta_1$ that $A_i(1) = a_i$, this is equivalent to defining the $A_i(m)$'s by the recurrence 
	\[
	\left[\begin{array}{c}
	A_1(m+1)\\
	A_2(m+1)\\
	\vdots\\
	A_b(m+1) \\
	A_{b+1}(m+1)
	\end{array}\right] \hspace{-2pt}=\hspace{-2pt} 
	\left[\begin{array}{cccc}
	a_1(bm+1) &  &  &  \\
	a_2(bm+1) & a_1(bm+2) &  &  \\
	\vdots & \vdots & \ddots &  \\
	a_b(bm + 1) & \cdots & a_1(bm+b) & \\
	a_1^2\binom{bm+1}{2} + a_{b+1}(bm+1) & a_b(bm+2) & \cdots & a_1(bm+b+1)
	\end{array}
	\right] \hspace{-4pt}
	\left[\begin{array}{c}
	A_1(m)\\
	A_2(m)\\
	\vdots\\
	A_b(m)\\
	A_{b+1}(m)
	\end{array}\right]
	\]
	with initial conditions 
	\[\left[\begin{array}{c}
	A_1(1)\\
	A_2(1)\\
	\vdots\\
	A_b(1) \\
	A_{b+1}(1)
	\end{array}\right] = \left[\begin{array}{c}
	a_1\\
	a_2\\
	\vdots\\
	a_b \\
	a_{b+1}
	\end{array}\right].\]
	In our computation of these coefficients, we use the ``$b$-tuple factorial'' notation, defined recursively by \[n!^{(b)} = (n-b)!^{(b)}n\]
	for $n \geq b$ and $n!^{(b)} = 1$ otherwise. 
	\begin{prop}\label{sumform}
		$A_i(m)$ is given by the following. For $\ell < b$,
		\begin{align*}
		A_\ell(m) &= a_ka_1^{m-1}(bm - b + \ell)!^{(b)} + \sum_{\ell > \alpha_1 > \cdots > \alpha_n > 0} \left(a_1^{m-n-1}a_{\ell-\alpha_1+1}a_{\alpha_n}\prod_{i=1}^{n-1}a_{\alpha_i - \alpha_{i-1} + 1}\right)\cdot\\
		&\sum_{i_1=1}^{m-1}\sum_{i_2=1}^{i_1-1}\cdots\sum_{i_n=1}^{i_{n-1}-1}\frac{(bm-b+\ell)!^{(b)}(bi_1 + \alpha_1)!^{(b)}\cdots(bi_n + \alpha_n)!^{(b)}}{(bi_1 + \ell)!^{(b)}(bi_2 + \alpha_1)!^{(b)}\cdots (bi_n + \alpha_{n-1})!^{(b)}}.
		\end{align*}
		For $\ell = b$,
		\begin{align*}
		A_b(m) &= a_ba_1^{m-1}\frac{(bm)!^{(b)}}{b} + \sum_{k > \alpha_1 > \cdots > \alpha_n > 0} \left(a_1^{m-n-1}a_{b-\alpha_1+1}a_{\alpha_n}\prod_{i=1}^{n-1}a_{\alpha_i - \alpha_{i-1} + 1}\right)\cdot\\
		&\sum_{i_1=1}^{m-1}\sum_{i_2=1}^{i_1-1}\cdots\sum_{i_n=1}^{i_{n-1}-1}\frac{(bm)!^{(b)}(bi_1 + \alpha_1)!^{(b)}\cdots(bi_n + \alpha_n)!^{(b)}}{(bi_1 + b)!^{(b)}(bi_2 + \alpha_1)!^{(b)}\cdots (bi_n + \alpha_{n-1})!^{(b)}}.
		\end{align*}
		And for $\ell = b+1$,
		\begin{align*}
		A_{b+1}(m) &= a_{b+1}a_1^{m-1}\frac{(bm +1)!^{(b)}}{bm+1} + a_1^{m+1}\frac{b}{2}\sum_{r=1}^{m-1}\frac{(bm+1)!^{(b)}(br+1)!^{(b)}}{(br+b+1)!^{(b)}}r \\&+ a_ba_2a_1^{m-2}\frac{1}{b}\sum_{r=1}^{m-1}\frac{(bm+1)!^{(b)}(br+b)!^{(b)}}{(br+b+1)!^{(b)}}\\
		&+ \sum_{\substack{b \geq \alpha_1 > \cdots > \alpha_n > 0 \\ \{\alpha_i\} \neq \{b\}}}\left(a_1^{m-n-1}a_{b-\alpha_1+2}a_{\alpha_n}\prod_{i=1}^{n-1}a_{\alpha_i - \alpha_{i-1} + 1}\right)\cdot\\
		&\sum_{i_1=1}^{m-1}\sum_{i_2=1}^{i_1-1}\cdots\sum_{i_n=1}^{i_{n-1}-1}\frac{(bm+1)!^{(b)}(bi_1 + \alpha_1)!^{(b)}\cdots(bi_n + \alpha_n)!^{(b)}}{(bi_1 + b+1)!^{(b)}(bi_2 + \alpha_1)!^{(b)}\cdots (bi_n + \alpha_{n-1})!^{(b)}}.
		\end{align*}
	\end{prop}
	\begin{proof}
		This is easily shown from the recursive descriptions of the $A_i(m)$ and the following observation from \cite[\S 1.2]{springer}, described in a similar form by \cite[Lemma 4]{swedish}: 
		\begin{lem}
			Let $f, g : \BN \to k$. Then the difference equation
			\[A(n+1) = f(n)A(n) + g(n)\]
			with initial condition $A(1) = a$ is uniquely satisfied by
			\[A(m) = \left[\prod_{j = 1}^{m-1} f(j)\right]a + \sum_{r=1}^{m-1}\left[\prod_{j = r+1}^{m-1}f(j)\right]g(r).\]
		\end{lem}
		Using the result for $A_1, \ldots, A_{k-1}$ for $g$ and applying the recursive definition above, we get the summation form for $A_k$ as desired. 
	\end{proof}
	Recall that all equalities are stated in $k$, and all the quantities in Proposition~\ref{sumform} other than the $a_i$'s are just integers which are reduced modulo $p$; they are in $\BF_p \subseteq k$, so we can compute them as elements of $\BF_p$. By taking an inverse of $b$ modulo $p$, we can simplify the expressions by making them include factorials instead of $b$-tuple factorials:
	\begin{lem}\label{factorial}
		Let $y \in \{0, 1, \ldots, b-1\}$ and $x \geq 0$. Then we have an equality of reductions modulo $p$
		\[(bx + y)!^{(b)} = b^x\frac{(x + yt)!}{(yt)!},\]
		where $t$ is the smallest lift of $b^{-1} \in \BF_p$ to $\BN$. 
	\end{lem}
	\begin{proof}
		In $\BF_p$ we can write
		\[(bx + y)!^{(b)} = (bx + y)(b(x-1) + y)\cdots(b + y).\]
		Multiplying by $t^x = (b^x)^{-1}$, we have
		\[(bx + y)!^{(b)} = b^x(x + yt)(x-1 + yt)\cdots (1 + yt) = b^x\frac{(x + yt)!}{(yt)!}\]
		as desired. 
	\end{proof}
	Applying Lemma~\ref{factorial} to the last part of Proposition~\ref{sumform}, we get
	\begin{prop}\label{coeff}
		The first nontrivial coefficient of $f^{p}$ is 
		\begin{align*} 
		A_{b+1}(p) &= a_{b+1}a_1^{p-1}b^p\frac{(p + t)!}{t!(b+1)} + a_1^{p+1}\frac{b}{2}\sum_{r=1}^{p-1}\frac{(p + t)!(r+t)!}{(r+1+t)!t!}b^{p-1}r + a_ba_2a_1^{p-2}\sum_{r=1}^{p-1}b^p\frac{(p+t)!(r+1)!}{(r+1+t)!}\\
		&+ b^{p-1}\hspace{-0.75cm}\sum_{b > \alpha_1 > \cdots > \alpha_n > 0} \left(a_1^{p-n-1}a_{b - \alpha_1 + 2}a_{\alpha_n}\prod_{i=1}^{n-1}a_{\alpha_i - \alpha_{i+1} +1}\right)\cdot\\&\qquad\qquad\qquad\qquad\,\sum_{i_1 = 1}^{p-1}\sum_{i_2 = 1}^{i_1 - 1} \cdots \sum_{i_n = 1}^{i_{n-1}-1}\frac{(p+t)!(i_1 + \alpha_1t)!\cdots(i_n + \alpha_nt)!}{(i_1 + 1 + t)!(i_2 + \alpha_1t)!\cdots(i_n + \alpha_{n-1}t)!(\alpha_nt)!}\\
		&+ b^{p-1}\hspace{-0.75cm}\sum_{\substack{b = \alpha_1 > \cdots > \alpha_n > 0\\ n > 1}} \left(a_1^{p-n-1}a_{2}a_{\alpha_n}\prod_{i=1}^{n-1}a_{\alpha_i - \alpha_{i+1} +1}\right)\cdot\\&\qquad\qquad\qquad\qquad\,\sum_{i_1 = 1}^{p-1}\sum_{i_2 = 1}^{i_1 - 1} \cdots \hspace{-2pt}\sum_{i_n = 1}^{i_{n-1}-1}\frac{(p+t)!(i_1 + \alpha_1t)!(i_2 + \alpha_2t)!\cdots(i_n + \alpha_nt)!}{(i_1 + 1 + t)!(i_2 + \alpha_1t)!(i_3 + \alpha_2t)!\cdots(i_n + \alpha_{n-1}t)!(\alpha_nt)!}.\\
		\end{align*}
	\end{prop}
	The goal of the rest of this section is to compute a more convenient form for these sums. We will frequently and without comment compute only residues mod $p$, passing to $\BQ_p$ and moving factors of $p$ in denominators around whenever necessary. 
	\begin{prop}\label{easyterms}
		\[b^p\frac{(p+t)!}{t!(b+1)!} = b^p\sum_{r=1}^{p-1}\frac{(p+t)!(r+1)!}{(r+1+t)!} = 0,\qquad \frac{b}{2}\sum_{r=1}^{p-1}\frac{(p+t)!(r+t)!}{(r+1+t)!t!}b^{p-1}r = \frac{b+1}{2}.\]
	\end{prop}
	\begin{proof}
		We start with the first expression. The fact that $b^p\frac{(p+t)!}{t!(b+1)!}$ vanishes in $\BF_p$ is obvious from the fact that $b+1 < p$. For the second expression, Wilson's Theorem gives us $(r+1)! = (-1)^r/(p-r-2)!$ for positive $r < p - 1$ and the $r = p-1, r < p-t-1$ terms are obviously zero so we have 
		\begin{align*}
		\sum_{r=1}^{p-1}\frac{(p+t)!(r+1)!}{(r+1+t)!} &= \sum_{r=p-t-1}^{p-2}(-1)^r\frac{(p+t)!}{(r+1+t)!(p-r-2)!}\\
		&= (p+t)\sum_{r=p-t-1}^{p-2}(-1)^r\binom{p+t-1}{r+t+1}\\
		&= t\sum_{i=0}^{t-1}(-1)^{i + p - t - 1}\binom{p + t - 1}{p+i}\\
		&= (-1)^tt\sum_{i=0}^{t-1}(-1)^{i}\binom{t - 1}{i}\\
		&= 0
		\end{align*}
		by the binomial theorem, as desired. The congruence $\binom{p + t - 1}{p+i} \equiv \binom{t-1}{i}$ modulo $p$ is evident either from writing them down as quotients and cancelling factors of $p$ or as a special case of Lucas's Theorem. 
		
		Finally, the only nonzero term in the third sum $\sum_{r=1}^{p-1}\frac{(p+t)!(r+t)!}{(r+1+t)!t!}b^{p-1}r$ occurs when $r + 1 + t = p$, i.e. $r = p - t - 1$. Therefore, 
		\begin{align*}
		\frac{b}{2}\sum_{r=1}^{p-1}\frac{(p+t)!(r+t)!}{(r+1+t)!t!}b^{p-1}r &= \frac{b}{2}\cdot\frac{(p+t)!(p-1)!}{(p)!t!}b^{p-1}(p-t-1)\\
		&= -\frac{b}{2}\cdot \frac{p!t!}{p!t!}b^{p-1}(p-t-1)\\
		&= \frac{b}{2}\cdot (t+1)\\
		&= \frac{b+1}{2}
		\end{align*}
		which proves the proposition.
	\end{proof}
	For the remaining two terms, we proceed by induction. From the form that we have reduced them to, it suffices to compute the sum
	\[\varphi_\beta(\alpha_1, \ldots, \alpha_n) := \sum_{i_1 = 1}^{p-1}\sum_{i_2 = 1}^{i_1 - 1} \cdots \sum_{i_n = 1}^{i_{n-1}-1}\frac{(p+t)!(i_1 + \alpha_1t)!\cdots(i_n + \alpha_nt + \beta)!}{(i_1 + 1 + t)!(i_2 + \alpha_1t)!\cdots(i_n + \alpha_{n-1}t)!(\alpha_nt)!} \in \BF_p\]
	where $\beta < b$ is a nonnegative integer, and $\alpha_1 > \ldots > \alpha_n$ are positive integers at most $b$. In fact, we will show that this expression has a much simpler form in terms of the inputs:
	\begin{prop}\label{induct}
		Let $\beta, \alpha_1, \ldots, \alpha_n$ be as above. Then 
		\[\varphi_\beta(\alpha_1, \ldots, \alpha_n) = (-1)^n\delta(\alpha_n, 1)\delta(\beta, 0)\]
		where $\delta(\cdot, \cdot)$ denotes the Kronecker delta. 
	\end{prop}
	\begin{proof}
		Let $n \geq 2$, $\beta < b$ and $b \geq \alpha_1 > \cdots > \alpha_n > 0$ be integers. We claim the result is true by induction on $n$. 
		
		Fix $i_{n-1} < p$. Then by cancelling factors of $p$ and analyzing the factors left over, we have that $\sum_{i_n=1}^{i_{n-1}-1} \frac{(i_n + \alpha_nt + \beta)!}{(i_n + \alpha_{n-1}t)!}$ reduces to
		\begin{align*}
		&\sum_{i_n = 1}^{i_{n-1}-1} \frac{(i_n + \alpha_nt + \beta + (\alpha_{n-1} - \alpha_n)p)!\lfloor\frac{i_n + \alpha_nt + \beta}{p}\rfloor!}{(i_n + \alpha_{n-1}t)!((\alpha_{n-1} - \alpha_n)p)!(\alpha_{n-1} - \alpha_n + 1) \cdots (\alpha_{n-1} - \alpha_n + \lfloor \frac{i_n + \alpha_nt + \beta}{p}\rfloor)}\\
		&= \sum_{i_n = 1}^{i_{n-1}-1}\binom{i_n + \alpha_nt + \beta + (\alpha_{n-1} - \alpha_n)p}{i_n + \alpha_{n-1}t}\cdot\\
		&\qquad\qquad\qquad\frac{\lfloor \frac{i_n + \alpha_nt +\beta}{p}\rfloor!((\alpha_{n-1} - \alpha_n)(p - t) + \beta)!}{((\alpha_{n-1} - \alpha_n)p)! (\alpha_{n-1} - \alpha_n + 1) \cdots (\alpha_{n-1} - \alpha_n + \lfloor \frac{i_n + \alpha_nt + \beta}{p}\rfloor)}.
		\end{align*}
		Of course, $\lfloor \frac{i_n + \alpha_nt + \beta}{p}\rfloor$ can only take on the values $\lfloor \frac{\alpha_nt + \beta}{p}\rfloor$ and $\lfloor \frac{\alpha_nt + \beta}{p}\rfloor + 1$ depending on whether $i_n$ is sufficiently large (since we know a priori that $i_n < p$). In particular, the former value is taken on if and only if $i_n < \lceil\alpha_nt + \beta\rceil_p - (\alpha_nt + \beta)$ (here $\lceil x\rceil_p$ denotes the least multiple of $p$ which is greater than $x$; note that this means $\lceil kp \rceil_p = (k+1)p$). If $i_{n-1} < \lceil \alpha_nt + \beta \rceil_p - (\alpha_nt  + \beta)$, then $i_n \leq i_{n-1} - 1 < \lceil\alpha_nt + \beta \rceil_p - (\alpha_nt + \beta) - 1$, so we have $\lfloor \frac{i_n + \alpha_nt + \beta}{p}\rfloor = \lfloor \frac{\alpha_nt + \beta}{p}\rfloor$ (in particular it doesn't depend on any of the indices in the sum). Therefore, in this case we have 
		\begin{align*}
		\sum_{i_n=1}^{i_{n-1}-1} \frac{(i_n + \alpha_nt + \beta)!}{(i_n + \alpha_{n-1}t)!} &= \frac{\lfloor\frac{\alpha_nt + \beta}{p}\rfloor!((\alpha_{n-1} - \alpha_n)(p-t) + \beta)!}{((\alpha_{n-1} - \alpha_n)p)!(\alpha_{n-1} - \alpha_n + 1)\cdots (\alpha_{n-1} - \alpha_n + \lfloor\frac{\alpha_nt + \beta}{p}\rfloor)}\cdot\\&\qquad\qquad\qquad\sum_{i_n=1}^{i_{n-1}-1}\binom{i_n + \alpha_nt + \beta + (\alpha_{n-1} - \alpha_n)p}{i_n + \alpha_{n-1}t}.
		\end{align*}
		By the hockey stick identity, the sum on the right hand side is
		\begin{align*}
		&\binom{i_{n-1} + \alpha_nt + \beta + (\alpha_{n-1} - \alpha_n)p}{i_{n-1} + \alpha_{n-1}t - 1} - \binom{1 + \alpha_nt + \beta + (\alpha_{n-1} - \alpha_n)p}{\alpha_{n-1}t}\\
		&= \frac{(i_{n-1} + \alpha_nt + \beta + (\alpha_{n-1} - \alpha_n)p)!}{(i_{n-1} + \alpha_{n-1}t - 1)!((\alpha_{n-1} - \alpha_n)(p - t) + \beta + 1)!} - \frac{(1 + \alpha_nt + 
			\beta + (\alpha_{n-1} - \alpha_n)p)!}{(\alpha_{n-1}t)!((\alpha_{n-1} - \alpha_n)(p-t) + \beta + 1)!}\\
		&= \frac{(i_{n-1} + \alpha_nt + \beta)!(\alpha_{n-1} - \alpha_n)p)!\cdot(\alpha_{n-1} - \alpha_n + 1)\cdots\left(\alpha_{n-1} - \alpha_n + \left\lfloor \frac{i_{n-1} + \alpha_nt + \beta}{p}\right\rfloor\right)}{(i_{n-1} + \alpha_nt - 1)!((\alpha_{n-1} - \alpha_n)(p - t) + \beta + 1)!\left\lfloor\frac{i_{n-1} + \alpha_nt + \beta}{p}\right\rfloor!}\\
		&\qquad\qquad-\frac{(1 + \alpha_nt + \beta)!(\alpha_{n-1} - \alpha_n)p)!(\alpha_{n-1} - \alpha_n + 1)\cdots\left(\alpha_{n-1} - \alpha_n + \left\lfloor\frac{1 + \alpha_nt + \beta}{p}\right\rfloor\right)}{(\alpha_{n-1}t)!((\alpha_{n-1} - \alpha_n)(p-t) + \beta + 1)!\left\lfloor\frac{1 + \alpha_nt + \beta}{p}\right\rfloor!}.\\
		\end{align*}
		Recall the assumption that $i_{n-1} < \lceil \alpha_nt +\beta \rceil_p - (\alpha_nt + \beta)$, from which it follows that $\lfloor\frac{i_{n-1} + \alpha_nt + \beta}{p}\rfloor = \lfloor\frac{\alpha_nt + \beta}{p}\rfloor$ and thus
		\begin{align}\label{case1}
		\sum_{i_n=1}^{i_{n-1}-1}\frac{(i_n + \alpha_nt + \beta)!}{(i_n + \alpha_{n-1}t)!} &= \frac{(i_{n-1} + \alpha_nt + \beta)!}{(i_{n-1} + \alpha_{n-1}t - 1)!(\beta + 1 - (\alpha_{n-1} - \alpha_n)t)} \\\nonumber &\qquad\qquad\qquad\qquad\qquad\qquad- \frac{(\alpha_nt + \beta + 1)!\ell(\alpha_{n-1}, \alpha_n)}{(\alpha_{n-1}t)!(\beta + 1 - (\alpha_{n-1} - \alpha_n)t)}
		\end{align}
		where $\ell(\alpha_{n-1}, \alpha_n)$ is defined to be $1$ if $\lfloor \frac{1 + \alpha_nt + \beta}{p}\rfloor = \lfloor \frac{\alpha_nt + \beta}{p}\rfloor$ (i.e. if $\alpha_nt + \beta \neq -1$ modulo $p$) and otherwise
		\[\ell(\alpha_{n-1}, \alpha_n) := \frac{\alpha_{n-1} - \alpha_n + \lfloor\frac{\alpha_nt + \beta}{p}\rfloor + 1}{\lfloor\frac{\alpha_nt + \beta}{p}\rfloor + 1}.\]
		
		In the remaining case, $i_{n-1} \geq \lceil \alpha_nt + \beta\rceil_p - (\alpha_nt + \beta)$, so $\sum_{i_n=1}^{i_{n-1}-1} \frac{(i_n + \alpha_nt + \beta)!}{(i_n + \alpha_{n-1}t)!}$ reduces to
		\begin{align}\label{case2}
		&\sum_{i_n=1}^{\lceil \alpha_nt + \beta\rceil_p - (\alpha_nt + \beta) - 1} \frac{(i_n + \alpha_nt + \beta)!}{(i_n + \alpha_{n-1}t)!} + \sum_{i_n=\lceil \alpha_nt + \beta\rceil_p - (\alpha_nt + \beta)}^{i_{n-1}-1} \frac{(i_n + \alpha_nt + \beta)!}{(i_n + \alpha_{n-1}t)!}\\
		&= \nonumber\frac{((\alpha_{n-1} - \alpha_n)(p - t) + \beta)!\lfloor\frac{\alpha_nt + \beta}{p}\rfloor!}{((\alpha_{n-1} - \alpha_n)p)!(\alpha_{n-1} - \alpha_n + 1) \cdots (\alpha_{n-1} - \alpha_n + \lfloor\frac{\alpha_nt + \beta}{p}\rfloor)}\cdot\\
		&\nonumber\qquad\qquad\qquad\qquad\Bigg(\sum_{i_n=1}^{\lceil \alpha_nt + \beta\rceil_p - (\alpha_nt +\beta)- 1}\binom{i_n + \alpha_nt +\beta + (\alpha_{n-1} - \alpha_n)t}{i_n + \alpha_{n-1}t}\\\nonumber&+ \frac{\lfloor\frac{\alpha_nt + \beta}{p}\rfloor + 1}{\alpha_{n-1} - \alpha_n + \lfloor\frac{\alpha_nt + \beta}{p}\rfloor + 1}\sum_{i_n=\lceil \alpha_nt + \beta\rceil_p - (\alpha_nt+\beta)}^{i_{n-1}-1}\binom{i_n + \alpha_nt + \beta+(\alpha_{n-1} - \alpha_n)t}{i_n + \alpha_{n-1}t}\Bigg).
		\end{align}
		
		Using the hockey stick identity, 
		\begin{align*}&\sum_{i_n=1}^{\lceil \alpha_nt + \beta \rceil_p - (\alpha_nt + \beta) - 1}\binom{i_n + \alpha_nt + \beta+ (\alpha_{n-1} - \alpha_n)p}{i_n + \alpha_{n-1}t} \\&\qquad\qquad\qquad\qquad\quad= \binom{\lceil\alpha_nt + \beta\rceil_p + (\alpha_{n-1} - \alpha_n)p}{\lceil\alpha_nt + \beta\rceil_p + (\alpha_{n-1} - \alpha_n)t - \beta - 1} - \binom{\alpha_nt + \beta + 1 + (\alpha_{n-1} - \alpha_n)p}{\alpha_{n-1}t}\end{align*}
		while
		\begin{align*}&\sum_{i_n=\lceil \alpha_nt + \beta\rceil_p - (\alpha_nt +\beta)}^{i_{n-1}-1}\binom{i_{n} + \alpha_nt + \beta + (\alpha_{n-1} - \alpha_n)p}{i_{n} + \alpha_{n-1}t} \\&\qquad\qquad\qquad\qquad\quad= \binom{i_{n-1} + \alpha_nt + \beta + (\alpha_{n-1} - \alpha_n)p}{i_{n-1} + \alpha_{n-1}t - 1} - \binom{\lceil\alpha_nt + \beta\rceil_p + (\alpha_{n-1} - \alpha_n)p}{\lceil\alpha_nt + \beta\rceil_p + (\alpha_{n-1} - \alpha_n)t - \beta - 1}.\end{align*}
		Thereofre, it suffices to compute these four binomial coefficients modulo $p$; note that the second one may be written as a specialization of the third to $i_{n-1} = 1$. As before, we can write that \[\binom{i_{n-1} + \alpha_nt + \beta + (\alpha_{n-1} - \alpha_n)p}{i_{n-1} + \alpha_{n-1}t-1}\] reduces to
		\[\frac{(i_{n-1} + \alpha_nt + \beta)!((\alpha_{n-1} - \alpha_n)p)!(\alpha_{n-1} - \alpha_n + 1)\cdots(\alpha_{n-1} - \alpha_n + \lfloor\frac{i_{n-1} + \alpha_nt + \beta}{p}\rfloor)}{(i_{n-1} + \alpha_{n-1}t - 1)!((\alpha_{n-1} - \alpha_n)(p - t) +\beta + 1)!\lfloor\frac{i_{n-1} + \alpha_nt +\beta}{p}\rfloor!}.\]
		Since $b^2 < p$, the smallest lift of the inverse of $b$ modulo $p$ must be greater than $b$. Hence, 
		\[\beta < b < t \leq (\alpha_{n-1} - \alpha_n)t\]
		so we can write $\beta + 1 \leq (\alpha_{n-1} - \alpha_n)t$. Using the same manipulations as before,
		\[\binom{\lceil\alpha_nt + \beta\rceil_p + (\alpha_{n-1} - \alpha_n)p}{\lceil\alpha_nt + \beta \rceil_p + (\alpha_{n-1} - \alpha_n)t -\beta- 1}\]
		therefore reduces to
		\begin{align*}\frac{((\alpha_{n-1} - \alpha_n)p)!(\lceil\alpha_nt + \beta\rceil_p)!(\alpha_{n-1} - \alpha_n + 1)\cdots(\alpha_{n-1} - \alpha_n + \lfloor\frac{\alpha_nt +\beta}{p}\rfloor + 1)\lfloor\frac{(\alpha_{n-1} - \alpha_n)t - \beta - 1}{p}\rfloor!}{\splitdfrac{(\lfloor\frac{\alpha_nt + \beta}{p}\rfloor + 1)!(\lceil\alpha_nt + \beta\rceil_p)!((\alpha_{n-1} - \alpha_n)t - \beta - 1)!(\lfloor\frac{\alpha_nt + \beta}{p}\rfloor+2)\cdots}{(\lfloor\frac{\alpha_nt + \beta}{p}\rfloor+1+\lfloor\frac{(\alpha_{n-1} - \alpha_n)t - \beta - 1}{p}\rfloor)((\alpha_{n-1} - \alpha_n)(p-t)+\beta + 1)!}}.\end{align*}
		Finally, 
		\[\binom{\alpha_nt +\beta + 1 + (\alpha_{n-1} - \alpha_n)p}{\alpha_{n-1}t}\]
		reduces to
		\[\frac{(\alpha_nt + \beta + 1)!((\alpha_{n-1} - \alpha_n)p)!(\alpha_{n-1} - \alpha_n + 1)\cdots(\alpha_{n-1} - \alpha_n + \lfloor\frac{1 + \alpha_nt}{p}\rfloor)}{(\alpha_{n-1}t)!((\alpha_{n-1} - \alpha_n)(p-t)+\beta + 1)!\lfloor\frac{1 + \alpha_nt}{p}\rfloor!}\]
		
		so we can write in this case by substituting our computations into (\ref{case2}),
		\begin{align*}&\sum_{i_n = 1}^{i_{n-1} - 1}\frac{(i_n + \alpha_nt + \beta)!}{(i_n + \alpha_{n-1}t)!} = \frac{((\alpha_{n-1} - \alpha_n)(p - t) + \beta)!\lfloor\frac{\alpha_nt + \beta}{p}\rfloor!}{((\alpha_{n-1} - \alpha_n)p)!(\alpha_{n-1} - \alpha_n + 1) \cdots (\alpha_{n-1} - \alpha_n + \lfloor\frac{\alpha_nt + \beta}{p}\rfloor)}\cdot\\
		&\Bigg(\binom{\lceil\alpha_nt + \beta\rceil_p + (\alpha_{n-1} - \alpha_n)p}{\lceil\alpha_nt + \beta\rceil_p + (\alpha_{n-1} - \alpha_n)t -\beta - 1} - \binom{\alpha_nt +\beta + 1 + (\alpha_{n-1} - \alpha_n)p}{\alpha_{n-1}t}\\&+ \frac{\lfloor\frac{\alpha_nt + \beta}{p}\rfloor + 1}{\alpha_{n-1} - \alpha_n + \lfloor\frac{\alpha_nt +\beta}{p}\rfloor + 1}\Bigg(\binom{i_{n-1} + \alpha_nt + \beta + (\alpha_{n-1} - \alpha_n)p}{i_{n-1} + \alpha_{n-1}t - 1}\\ &\qquad\qquad\qquad\qquad\qquad\qquad\qquad\qquad\qquad\qquad -\binom{\lceil\alpha_nt + \beta\rceil_p + (\alpha_{n-1} - \alpha_n)p)}{\lceil\alpha_nt+\beta\rceil_p + (\alpha_{n-1} - \alpha_n)t -\beta - 1}\Bigg)\Bigg)\\
		&= \frac{(\alpha_1 - \alpha_2)\lfloor\frac{(\alpha_{n-1} - \alpha_n)t - \beta - 1}{p}\rfloor!}{\splitdfrac{((\alpha_{n-1} - \alpha_n)(p - t) + \beta + 1)(\lfloor\frac{\alpha_nt+\beta}{p}\rfloor+1)((\alpha_{n-1} - \alpha_n)t - \beta - 1)!(\lfloor\frac{\alpha_nt + \beta}{p}\rfloor+2)}{\cdots(\lfloor\frac{\alpha_nt + \beta}{p}\rfloor+1+\lfloor\frac{(\alpha_{n-1} - \alpha_n)t - \beta - 1}{p}\rfloor)}}\\
		&-\frac{\ell(\alpha_{n-1}, \alpha_n)}{((\alpha_{n-1} - \alpha_n)(p - t)+\beta + 1)}\frac{( \alpha_nt + \beta + 1)!}{(\alpha_{n-1}t)!}+\frac{(i_{n-1} + \alpha_nt + \beta)!}{(i_{n-1} + \alpha_{n-1}t - 1)!((\alpha_{n-1} - \alpha_n)(p - t)+\beta + 1)}.
		\end{align*}
		Recall that $\varphi_\beta(\alpha_1, \ldots, \alpha_n)$ is defined by
		\[\varphi_\beta(\alpha_1, \ldots, \alpha_n) := \sum_{i_1 = 1}^{p-1}\cdots\sum_{i_n=1}^{i_{n-1}-1}\frac{(p+t)!(i_1 + \alpha_1t)!\cdots(i_n + \alpha_nt + \beta)!}{(i_1 + 1 + t)!(i_2 + \alpha_1t)!\cdots(i_n + \alpha_{n-1}t)!(\alpha_n)!},\]
		and that when $\beta = 0$ this is the last kind of coefficient we want to compute (see Proposition~\ref{coeff}).
		Since the last two terms in our expression for $\sum_{i_n = 1}^{i_{n-1}-1} \frac{(i_n + \alpha_nt + \beta)!}{(i_n + \alpha_{n-1}t)!}$ are the same modulo $p$ as entire expression in the case $i_{n-1} < \lceil\alpha_nt + \beta\rceil_p -(\alpha_nt +\beta)$ (see (\ref{case1})), we can conclude that 
		\begin{align*}
		\varphi_\beta(\alpha_1, \ldots, \alpha_n) &= c_1\sum_{i_1 = 1}^{p-1}\cdots\sum_{i_{n-1} = \lceil\alpha_nt + \beta\rceil_p - (\alpha_n+\beta)}^{i_{n-2} - 1}\frac{(p + t)!(i_1 + \alpha_1t)!\cdots(i_{n-1} + \alpha_{n-1}t)!}{\splitdfrac{(i_1 + 1 + t)!(i_2 + \alpha_1t)!\cdots(i_{n-1} + \alpha_{n-2}t)!(\alpha_{n}t)!\cdot}{((\alpha_{n-1} - \alpha_n)t-\beta-1)!}}\\
		&+c_2\sum_{i_1 = 1}^{p-1}\cdots\sum_{i_{n-1} = 1}^{i_{n-2} - 1}\frac{(p + t)!(i_1 + \alpha_1t)!\cdots(i_{n-1} + \alpha_{n-1}t)!(\alpha_nt+\beta + 1)!}{(i_1 + 1 + t)!(i_2 + \alpha_1t)!\cdots(i_{n-1} + \alpha_{n-2}t)!(\alpha_{n}t)!(\alpha_{n-1}t)!}\\
		&+\frac{1}{1 + \beta - (\alpha_{n-1} - \alpha_{n})t}\cdot\\&\qquad\qquad\sum_{i_1 = 1}^{p-1}\cdots\sum_{i_{n-1} = 1}^{i_{n-2} - 1}\frac{(p + t)!(i_1 + \alpha_1t)!\cdots(i_{n-1} + \alpha_{n-1}t)!(i_{n-1} + \alpha_nt + \beta)!}{(i_1 + 1 + t)!(i_2 + \alpha_1t)!\cdots(i_{n-1} + \alpha_{n-2}t)!(\alpha_{n}t)!(i_{n-1} + \alpha_{n-1}t - 1)!}.
		\end{align*}
		where $c_1, c_2 \in \BF_p^*$ (i.e. they have no factors of $p$ in numerator or denominator) and they depend only on the $\alpha_i$'s. 
		Note that the terms of the first sum are of the form
		\begin{align*}
		\frac{(p + t)!(i_1 + \alpha_1t)!\cdots(i_{n-1} + \alpha_{n-1}t)!}{(i_1 + 1 + t)!(i_2 + \alpha_1t)!\cdots(i_{n-1} + \alpha_{n-2}t)!(\alpha_{n}t)!((\alpha_{n-1} - \alpha_n)t-\beta - 1)!} = c\frac{(i_{n-1} + \alpha_{n-1}t)!}{(\alpha_nt)!((\alpha_{n-1} - \alpha_n)t - \beta - 1)!}
		\end{align*}
		where $c$ has a nonnegative number of factors of $p$. In fact, we have
		\[\frac{(i_{n-1} + \alpha_{n-1}t)!}{(\alpha_nt)!((\alpha_{n-1} - \alpha_n)t - \beta - 1)!} = \frac{(i_{n-1} + \alpha_nt + \beta + (\alpha_{n-1} - \alpha_n)t - \beta)!}{(\alpha_nt)!((\alpha_{n-1} - \alpha_n)t -\beta - 1)!}\]
		and from the fact that $i_{n-1} + \alpha_nt + \beta \geq \lceil\alpha_nt + \beta\rceil_p$ it is immediate that this is zero modulo $p$. Therefore, each term of the first sum is zero, and we can ignore it altogether. The second sum is equal to
		\[c_2\frac{(\alpha_nt + \beta + 1)!}{(\alpha_nt)!}\sum_{i_1 = 1}^{p-1}\cdots\sum_{i_{n-1} = 1}^{i_{n-2} - 1}\frac{(p + t)!(i_1 + \alpha_1t)!\cdots(i_{n-1} + \alpha_{n-1}t)!}{(i_1 + 1 + t)!(i_2 + \alpha_1t)!\cdots(i_{n-1} + \alpha_{n-2}t)!(\alpha_{n-1}t)!},\]
		which is
		\[c_2\frac{(\alpha_nt + \beta + 1)!}{(\alpha_nt)!}\varphi_0(\alpha_1, \ldots, \alpha_{n-1}) = 0\]
		by induction (since $\alpha_{n-1} > \alpha_n \geq 1$). Therefore, $\varphi_\beta(\alpha_1, \ldots, \alpha_n)$ is
		\begin{align*} &\,\,\,\,\,\,\,\,\frac{1}{1 - (\alpha_{n-1} - \alpha_{n})t}\sum_{i_1 = 1}^{p-1}\cdots\sum_{i_{n-1} = 1}^{i_{n-2} - 1}\frac{(p + t)!(i_1 + \alpha_1t)!\cdots(i_{n-1} + \alpha_{n-1}t)!(i_{n-1} + \alpha_nt + \beta)!}{(i_1 + 1 + t)!(i_2 + \alpha_1t)!\cdots(i_{n-1} + \alpha_{n-2}t)!(\alpha_{n}t)!(i_{n-1} + \alpha_{n-1}t - 1)!}\\
		&= \frac{1}{1 - (\alpha_{n-1} - \alpha_{n})t}\varphi_{\beta + 1}(\alpha_1, \ldots, \alpha_{n-2}, \alpha_n) - \varphi_{\beta}(\alpha_1, \ldots, \alpha_{n-2}, \alpha_n)\\
		&= -\varphi_{\beta}(\alpha_1, \ldots, \alpha_{n-2}, \alpha_n)
		\end{align*}
		by induction since $\beta + 1 > 0$. 
		
		The base case ($n = 1$) also lends itself well to our methods: if $\alpha > 1$ or if $\beta > 0$, then $i_1 + \alpha t + \beta \geq i_1 + 1 + t$ and we can write
		\begin{align*}\varphi_\beta(\alpha) &= \sum_{i_1 = 1}^{p-1}\frac{(p + t)!(i_1 + \alpha t + \beta)!}{(i_1 + 1 + t)!(\alpha t)!} = \frac{(p+t)!((\alpha  - 1)t+\beta-1)!}{(\alpha t)!}\sum_{i_1 = 1}^{p-1}\binom{i_1 + \alpha t + \beta}{i_1 + 1 + t}\\
		&= \frac{(p+t)!((\alpha  - 1)t+\beta-1)!}{(\alpha t)!}\left(\binom{p + \alpha t + \beta}{p + t} - \binom{1 + \alpha t + \beta}{1 + t}\right)\\
		&= 0.
		\end{align*}
		On the other hand, if $\alpha = 1$ and $\beta = 0$, we find that 
		\[\varphi_\beta(\alpha) = \sum_{i_1 = 1}^{p-1}\frac{(p + t)!(i_1 + t)!}{(i_1 + 1 + t)!(t)!},\]
		and the only nonzero term in this sum occurs when $i_1 + 1 + t = p$. Thus, 
		\[\varphi_\beta(\alpha) = \frac{(p+t)!(p-1)!}{p!t!} = -1.\]
		It follows by induction that $\varphi_\beta(\alpha_1, \ldots, \alpha_n) = 0$ if $\alpha > 1$ or $\beta > 0$, and otherwise it is $(-1)^n$.
	\end{proof}
	Applying Proposition~\ref{easyterms} and Proposition~\ref{induct} to the sums achieved in Proposition~\ref{coeff}, we obtain
	\begin{equation}\label{intermediate} A_{b+1}(p) = \frac{b+1}{2}a_1^{p+1} + \sum_{n=1}^b(-1)^na_1^{p-n}\sum_{b \geq \alpha_1 > \cdots > \alpha_n = 1}a_{b-\alpha_1+2}\prod_{i=1}^{n-1}a_{\alpha_i - \alpha_{i+1} + 1}.\end{equation}
	A careful analysis of the terms in this sum yields the main theorem of this paper.
	\begin{proof}[Proof of Theorem~\ref{mainresult}]
		Recall from Corollary~\ref{sencor} that $A_{b+1}(p)$ is the first nonzero coefficient of $\Delta_{p}$, and that $A_{b+1}(p)$ is the $X^{bp + b+1}$-coefficient of $\Delta_p$. By Corollary~\ref{d_p}, it follows that $i_1(f) = bp +b$ if and only if $A_{b+1}(p) \neq 0$. By Theorem~\ref{hammer}, $f$ is $b$-ramified if and only if $i_1(f) = bp + b$, so in fact $f$ is $b$-ramified if and only if $A_{b+1}(p) \neq 0$.
		
		Since $a_1 \neq 0$, we can multiply (\ref{intermediate}) by $2a_1^{b-p}$  to get that $f$ is $b$-ramified if and only if 
		\[(b+1)a_1^{b+1} + 2\sum_{n=1}^b(-1)^na_1^{b-n}\sum_{b \geq \alpha_1 > \cdots > \alpha_n = 1}a_{b-\alpha_1+2}\prod_{i=1}^{n-1}a_{\alpha_i - \alpha_{i+1} + 1} \neq 0\]
		
		Fixing $n$, the term $a_2^{e_2} \cdots a_{b+1}^{e_{b+1}}$ appears exactly $\binom{e_2 + \cdots + e_{b+1}}{e_2, \ldots, e_{b+1}}$ times. Since $n = e_2 + \cdots + e_{b+1}$ (in particular it is determined by the $e_i$'s), any monomial can only appear for at most one value of $n$. In fact, $a_2^{e_2}\cdots a_{b+1}^{e_{b+1}}$ shows up if and only if $e_2 + 2e_3 + \cdots + be_{b+1} = b$. In this case, \[e_1 = b - n = b - (e_2 + \cdots + e_{b+1})\] so this condition is equivalent to $a_1^{e_1}\cdots a_{b+1}^{e_{b+1}}$ appearing in the entire sum if and only if 
		\[e_1 + 2e_2 + \cdots + (b+1)e_{b+1} = 2b\]
		and \[e_1 + e_2 + \cdots + e_{b+1} = b.\] It follows that $f$ is $b$-ramified if and only if 
		\[(b+1)a_1^{b+1} + \sum_{\substack{e_1, \ldots, e_{b+1} \geq 0\\e_1 + \cdots + e_{b+1} = b\\e_1 + 2e_2 + \cdots + (b+1)e_{b+1} = 2b}}2(-1)^{e_2 + \cdots + e_{b+1}}\binom{e_2 + \cdots + e_{b+1}}{e_2, \ldots, e_{b+1}}a_1^{e_1}\cdots a_{b+1}^{e_{b+1}} \neq 0.\]
		We have finally recovered the polynomial (see Definition~\ref{P}) $P_b(a_1, \ldots, a_{b+1})$ as the criterion for $f$ being $b$-ramified. Adding in the assumption that $a_1 \neq 0$ from the beginning of this section (which is implied by $i(f) = b$), the desired result is established. 
	\end{proof}
	\section{Further Work and Applications to nonarchimedean Dynamics}\label{dynamics}
	The case where $k$ is perfect, so that $\mathcal{N}(k)$ is the group of wild automorphisms of the local field $k((X))$, is a natural application of Theorem~\ref{mainresult}. Our results are also of interest for their application to the case where $k$ is equipped with a nonarchimedean valuation, and the power series over $k$ are considered as acting on the open unit disc in $k$. In this section, $k$ is taken to be a nonarchimedean field of characteristic $p$ with absolute value $|\cdot|$.
	
	The ramification numbers of a power series, as we have defined them in Definition~\ref{ramification}, are only originally defined as being naturally an invariant of the corresponding automorphism of $k((X))$. However, in \cite{lindahl_l} and \cite{generic}, Lindahl and Rivera-Letelier found a deep connection between the ramification numbers of a power series over an arbitrary nonarchimedean field of characteristic $p$ and the location of its periodic points. In particular, if $k$ is a nonarchimedean field of characteristic $p$, then the power series \[f(z) \in z\left(1 + z\mathcal{O}_k[[z]]\right)\] converge on the open unit disc $\mathfrak{m}_k$, and all of their periodic points have minimal period of the form $p^n$ (see \cite[Lemma~2.1]{lindahl_l}). The relevant general lemma is from \cite[Lemma~2.4]{generic}:
	\begin{lem}\label{bigbrain}
		Suppose $f(z) \in \mathcal{N}(k)$ with $i_n(f) < \infty$ for all $n \geq 0$, and let $z_0 \in \mathfrak{m}_k$ be a periodic point under the action of $f$ of minimal period $p^n$. Then 
		\[|z_0| \geq \left|\frac{\delta_n(f)}{\delta_{n-1}(f)}\right|^{\frac{1}{p^n}},\]
		where $\delta_n(f)$ denotes the $X^{i_n(f)+1}$-coefficient of $f^{p^n}$.
	\end{lem}
	Using this fact and a computation of the relevant coefficient of $f^{p^n}$, Lindahl and Nordqvist (see \cite[Theorem A]{periodicpts}) proved a lower bound on the norm of periodic points under the action of $2$-ramified power series:
	\begin{theorem}[Lindahl and Nordqvist, 2018]
		Suppose $p \geq 5$ and let \[f(z) = z + \sum_{i = 1}^\infty a_{i}z^{i+2} \in \mathcal{O}_k[[z]].\] If $z_0 \in \mathfrak{m}_k$ is a periodic point of period $p^n$ under the action of $f$, then 
		\[|z_0| \geq \left|a_1^{p-3}\left(\frac{3}{2}a_1^3 + a_2^2 - a_1a_3\right)\right|^{\frac{1}{p}}.\]
	\end{theorem}
	The bound for $|z_0|$ comes from the computation of the first nontrivial coefficient of $f^{p^n}$, and it turns out to not depend on $n$ at all. Since the bound is a nonzero scalar multiple of the absolute value of $P_2$, it is only nontrivial so long as $f$ is $2$-ramified. Based on the the independence on $n$ in that case, we conjecture that the analogous bound is true for $b$-ramified power series.
	\begin{conjecture}\label{bigconj}
		Let $b \in \BN$ and suppose that $p > b^2$ is a prime. Moreover, let 
		\[f = z + \sum_{i=1}^\infty a_iz^{i+b}\]
		be a formal power series with coefficients in $\mathcal{O}_k$, where $a_1 \neq 0$. If $z_0 \in \mathfrak{m}_k$ is a periodic point of period $p^n$ under the action of $f$, then 
		\[|z_0| \geq \left|\frac{a_1^{p-b-1}}{2}P_b(a_1, \ldots, a_{b+1})\right|^{\frac{1}{p}}.\]
	\end{conjecture}
	
	Our computation of the relevant coefficient of $f^p$ in (\ref{intermediate}), combined with Lemma~\ref{bigbrain} gives a proof of Conjecture~\ref{bigconj} in the case $n=1$. For the general case, the appropriate coefficient of $f^{p^n}$ must be computed as well\footnote{
		Since we posted on arXiv.org a draft of this paper containing this conjecture, Jonas Nordqvist and Juan Rivera-Letelier informed us in a private communication that they have developed the required computation and therefore proved Conjecture~\ref{bigconj}.
	}. After conjugating by a translation, scaling, and taking iterates if necessary (see the remarks in \cite[p. 2]{generic}), this conjecture implies via the invariance of ramification under conjugation that any periodic point of a $b$-ramified power series over a nonarchimedean field of characteristic $p > b^2$ is isolated. That would add additional cases under which the main theorem of \cite{generic} holds, which represents progress towards \cite[Conjecture 1.2]{lindahl_l}. 
	\section*{Acknowledgments} The results in this paper originated from a project at PROMYS 2016. We are deeply
	grateful to Laurent Berger and Sandra Rozensztajn for proposing this problem, and for their mentoring. We also thank our PROMYS research lab counselor Alyosha Latyntsev for his support, as well as
	Krishanu Sankar, David Fried, Glenn Stevens, the PROMYS Foundation, and the Clay Mathematics
	Institute for making this research possible. In their most general form, many of the computations in this paper were completed at the University of Chicago REU during the summer of 2018. Therefore, the first author thanks his mentor Drew Moore for the encouragement he gave, as well as J.P. May for making the REU possible. Finally, we would like to thank Jonas Nordqvist and Juan Rivera-Letelier for their helpful correspondence with us regarding the hypotheses and implications of Conjecture~\ref{bigconj}. 
	
	\bibliographystyle{plainurl}
	\bibliography{main.bbl}

\begin{thebibliography}{10}

\bibitem{finite3}
Frauke~M Bleher, Ted Chinburg, Bjorn Poonen, and Peter Symonds.
\newblock Automorphisms of {H}arbater--{K}atz--{G}abber curves.
\newblock {\em Mathematische Annalen}, 368(1-2):811--836, 2017.

\bibitem{subgroups}
Rachel Camina.
\newblock Subgroups of the {N}ottingham group.
\newblock {\em Journal of Algebra}, 196(1):101--113, 1997.

\bibitem{book}
Marcus du~Sautoy, Dan Segal, and Aner Shalev.
\newblock {\em New Horizons in pro-$p$ Groups}.
\newblock New Horizons in Pro-$p$ Groups. Birkh{\"a}user Boston, 2000.

\bibitem{springer}
Saber Elaydi.
\newblock {\em An Introduction to Difference Equations}.
\newblock Springer Science \& Business Media, 2005.

\bibitem{finite1}
Sandrine Jean.
\newblock Conjugacy classes of series in positive characteristic and {W}itt
  vectors.
\newblock {\em Journal de Th{\'e}orie des Nombres de Bordeaux}, 21(2):263--284,
  2009.

\bibitem{inspiration}
S.A. Jennings.
\newblock Substitution groups of formal power series under substitution.
\newblock {\em Canadian Journal of Mathematics}, 6:325 -- 340, 1954.

\bibitem{johnson}
D.~L. Johnson.
\newblock The group of formal power series under substitution.
\newblock {\em Journal of the Australian Mathematical Society. Series A -- Pure
  Mathematics and Statistics}, 45(3):296--302, 04 2009.

\bibitem{keating}
Kevin Keating.
\newblock Automorphisms and extensions of {$k((t))$}.
\newblock {\em Journal of Number Theory}, 41(3):314 -- 321, 1992.

\bibitem{klopsch}
Benjamin Klopsch.
\newblock Automorphisms of the {N}ottingham group.
\newblock {\em Journal of Algebra}, 223(1):37 -- 56, 2000.

\bibitem{laubie_saine}
F~Laubie and M~Sa{\"i}ne.
\newblock Ramification of automorphisms of {$k((t))$}.
\newblock {\em Journal of Number Theory}, 63(1):143 -- 145, 1997.

\bibitem{dynamic1}
Fran{\c{c}}ois Laubie, Abbas Movahhedi, and Alain Salinier.
\newblock Syst{\`e}mes dynamiques non archim{\'e}diens et corps des norms
  (non-archimedean dynamic systems and fields of norms).
\newblock {\em Compositio Mathematica}, 132(1):57--98, 2002.

\bibitem{lindahl}
Karl-Olof Lindahl.
\newblock The size of quadratic $p$-adic linearization disks.
\newblock {\em Advances in Mathematics}, 248:872 -- 894, 2013.

\bibitem{periodicpts}
Karl-Olof Lindahl and Jonas Nordqvist.
\newblock Geometric location of periodic points of 2-ramified power series.
\newblock {\em Journal of Mathematical Analysis and Applications}, 2018.

\bibitem{lindahl_l}
Karl-Olof Lindahl and Juan Rivera-Letelier.
\newblock Optimal cycles in ultrametric dynamics and minimally ramified power
  series.
\newblock {\em Compositio Mathematica}, 152(1):187--222, 009 2015.

\bibitem{generic}
Karl-Olof Lindahl and Juan Rivera-Letelier.
\newblock Generic parabolic points are isolated in positive characteristic.
\newblock {\em Nonlinearity}, 29(5):1596, 2016.

\bibitem{dynamic2}
Jonathan Lubin.
\newblock Nonarchimedean dynamical systems.
\newblock {\em Compositio Mathematica}, 94(3):321--346, 1994.

\bibitem{lubin_sen}
Jonathan Lubin.
\newblock Sen's theorem on iteration of power series.
\newblock {\em Proceedings of the American Mathematical Society},
  123(1):63--66, 1995.

\bibitem{finite2}
Jonathan Lubin.
\newblock Torsion in the {N}ottingham group.
\newblock {\em Bulletin of the London Mathematical Society}, 43(3), 2011.

\bibitem{swedish}
Jonas Nordqvist.
\newblock Characterization of 2-ramified power series.
\newblock {\em Journal of Number Theory}, 174:258--273, 2017.

\bibitem{letelier}
J.~Rivera-Letelier.
\newblock {\em Dynamique des fonctions rationelles sur des corps locaux}.
\newblock PhD thesis, Universite de Paris XI, 2000.

\bibitem{sen}
Shankar Sen.
\newblock On automorphisms of local fields.
\newblock {\em Annals of Mathematics}, 90(1):33--46, 1969.

\bibitem{serre}
Jean-Pierre Serre.
\newblock {\em Galois cohomology}.
\newblock Springer Science \& Business Media, 2013.

\bibitem{york2}
I.O. York.
\newblock The exponent of certain $p$-groups.
\newblock {\em Proceedings of the Edinburgh Mathematical Society}, 33:483--490,
  1990.

\bibitem{york1}
I.O. York.
\newblock {\em The group of formal power series under substitution}.
\newblock PhD thesis, Nottingham University, 1990.

\end{thebibliography}
\end{document}